\documentclass[12pt,leqno,a4paper]{article}
\usepackage{amsthm}
\usepackage{indentfirst}
\usepackage{amsmath}
\usepackage{mathrsfs}
\usepackage{txfonts}
\usepackage{paralist}
\usepackage{amssymb,enumerate}
\usepackage{hyperref}
\usepackage{titlesec}
\usepackage{url}
\usepackage{extarrows}
\usepackage{booktabs}
\usepackage{threeparttable}
\usepackage{color}

\renewcommand\thesubsection{\thesection.\Alph{subsection}}
\newcommand{\sbnal}{0.3em}   
\titleformat{\subsection}[runin]{\bfseries}{\thesubsection}{\sbnal}{}[]

\theoremstyle{theorem}
\newtheorem{mainthm}{Theorem}
\newtheorem*{conj}{Conjecture}
\newtheorem{thm}{Theorem}[section]
\newtheorem{prop}[thm]{Proposition}
\newtheorem{lem}[thm]{Lemma}
\newtheorem{cor}[thm]{Corollary}

\theoremstyle{definition}

\newtheorem{defn}[thm]{Definition}
\newtheorem{rem}[thm]{Remark}

\numberwithin{equation}{section}

\newcommand{\GL}{\operatorname{GL}}
\newcommand{\GU}{\operatorname{GU}}
\newcommand{\U}{\operatorname{U}}
\newcommand{\SL}{\operatorname{SL}}
\newcommand{\SU}{\operatorname{SU}}
\newcommand{\PSL}{\operatorname{PSL}}
\newcommand{\PSU}{\operatorname{PSU}}
\newcommand{\Sp}{\operatorname{Sp}}
\newcommand{\PSp}{\operatorname{PSp}}

\newcommand{\Aut}{\operatorname{Aut}}

\newcommand{\Ind}{\operatorname{Ind}}
\newcommand{\Res}{\operatorname{Res}}
\newcommand{\Irr}{\operatorname{Irr}}
\newcommand{\bl}{\operatorname{bl}}
\newcommand{\IBr}{\operatorname{IBr}}

\newcommand{\dz}{\operatorname{dz}}
\newcommand{\Alp}{\operatorname{Alp}}

\newcommand{\rO}{\operatorname{O}}
\newcommand{\J}{\operatorname{J}}
\newcommand{\Lin}{\operatorname{Lin}}

\newcommand{\group}[1]{\langle#1\rangle}

\newcommand{\F}{\mathbb{F}}
\newcommand{\barF}{\overline{\mathbb{F}}}

\newcommand{\tG}{\tilde{G}}
\newcommand{\tH}{\tilde{H}}
\newcommand{\tK}{\tilde{K}}
\newcommand{\tbG}{\tilde{\mathbf{G}}}
\newcommand{\hG}{\hat{G}}

\newcommand{\tbL}{\tilde{\mathbf{L}}}
\newcommand{\tR}{\tilde{R}}

\newcommand{\tC}{\tilde{C}}

\newcommand{\tN}{\tilde{N}}

\newcommand{\tB}{\tilde{B}}

\newcommand{\eps}{\eta}

\newcommand{\cF}{\mathcal{F}}

\newcommand{\fZ}{\mathfrak{Z}}

\newcommand{\fS}{\mathfrak{S}}

\newcommand{\ttheta}{\tilde{\theta}}

\newcommand{\tDelta}{\tilde{\Delta}}

\newcommand{\tvarphi}{\tilde{\varphi}}
\newcommand{\tchi}{\tilde{\chi}}
\newcommand{\tpsi}{\tilde{\psi}}

\newcommand{\cW}{\mathrm{Alp}}

\newcommand{\hz}{\hat{z}}

\newcommand{\hchi}{\hat{\chi}}

\newcommand{\tcB}{\tilde{\mathcal{B}}}
\newcommand{\tOmega}{\tilde{\Omega}}

\let\Ga=\Gamma
\let\al=\alpha

\let\la=\lambda

\begin{document}

\title{On the inductive blockwise Alperin weight condition for type $\mathsf A$
\footnote{Supported by the NSFC (No.~11631001, No.~11901028, No.~11901478).}}

\author{Zhicheng Feng
	\footnote{School of Mathematics and Physics, University of Science and Technology Beijing, Beijing 100083, China. E-mail address: zfeng@pku.edu.cn} \quad Conghui Li
	\footnote{School of Mathematics, Southwest Jiaotong University, Chengdu 611756, China. E-mail address: liconghui@swjtu.edu.cn}
	\quad Jiping Zhang
	\footnote{School of Mathematical Sciences, Peking University, Beijing 100871, China. E-mail address: jzhang@pku.edu.cn}
}

\maketitle

\begin{abstract}
In this paper we prove the blockwise Alperin weight conjecture for finite special linear and unitary groups, for finite groups with abelian Sylow~$3$-subgroups,
 and verify the inductive blockwise Alperin weight condition for certain cases of groups of type $\mathsf A$.
We also give a classification for the~2-blocks of special linear and unitary groups.
\end{abstract}



\section{Introduction}\label{sect:intro}

The famous Alperin weight conjecture relates for a prime $\ell$ information about irreducible $\ell$-Brauer characters of a finite group $G$ to properties of $\ell$-local subgroups of $G$. 
We call an \emph{$\ell$-weight} a pair $(R,\varphi)$,
where $R$ is an $\ell$-subgroup of $G$ and
$\varphi\in\Irr(N_G(R))$ with $R\subseteq\ker\varphi$ is of $\ell$-defect zero viewed as a character of $N_G(R)/R$.
When such a character $\varphi$ exists,
$R$ is necessarily a radical $\ell$-subgroup of $G$~(\emph{i.e.} $R=\mathrm O_\ell(N_G(R))$).
For an $\ell$-block $B$ of $G$,
a weight $(R,\varphi)$ is called a \emph{$B$-weight} if $\bl(\varphi)^G=B$,
where $\bl(\varphi)$ is the $\ell$-block of $N_G(R)$ containing $\varphi$.
We denote by $\cW_\ell(B)$ the set of all $G$-conjugacy classes of $B$-weights. 
In \cite[p.~371]{Al87}, J. L. Alperin states the following (blockwise Alperin weight) conjecture.

\begin{conj}[Alperin, 1986]
	\label{weiconj}
	Let $G$ be a finite group, $\ell$ a prime.
	If $B$ is an $\ell$-block of $G$, then $|\cW_\ell(B)|=|\IBr_\ell(B)|$.
\end{conj}

While the Alperin weight conjecture was subsequently checked for
several families of finite groups, a significant breakthrough in the
case of general groups was achieved by  Navarro and Tiep \cite{NT11} in 2011,
where they reduced the block-free version of Alperin's weight conjecture to a question on simple groups.
In 2013, the blockwise version was also reduced by Sp\"ath \cite{Sp13}.
The blockwise Alperin weight conjecture holds
for all finite groups at the prime $\ell$, if all finite non-abelian simple groups satisfy
the so-called \emph{inductive blockwise Alperin weight (iBAW) condition} at the prime $\ell$.

This paper is a continuation of our previous paper \cite{FLZ20}.
In that paper, the inductive condition for the block-free version of Alperin weight conjecture from~\cite{NT11} has been verified for all simple groups of type $\mathsf A$.

We also focus on simple groups  of type $\mathsf A$ in this paper.
Let $q=p^f$ be a power of a prime $p$, $\eps=\pm 1$ and $\SL_n(\eps q)$  be the finite special linear or unitary group.
Here, $\SL_n(-q)$ is understood as $\SU_n(q)$.
We also denote $\PSL_n(\eps q)=\SL_n(\eps q)/Z(\SL_n(\eps q))$.
We remark that some progress has been made for type $\mathsf A$; for example, the (iBAW) condition has been verified for $\PSL_n(\eps q)$ with $\gcd(n,q-\eps)=1$ in \cite{LZ18, LZ19},
for unipotent blocks of $\SL_n(\eps q)$ with $\ell\nmid q \cdot\gcd(n,q-\eps)$ in \cite{Feng19},
and for blocks of $\SL_n(\eps q)$ with abelian defect groups when $\ell\nmid 3q(q-\eps)$ by Brough and Sp\"ath \cite{BS19a}.
For a list  about other types and more cases for which the inductive conditions has been verified, see for example \cite{FLZ20}  and the references therein.

Using the results in \cite{FLZ20}, we will verify the (iBAW) condition for a system of certain blocks of $\SL_n(\eps q)$; see Theorem \ref{iBAW-SLSU-some-special-case}.
We give some consequences of this result as follows. 
First, we achieve the blockwise Alperin weight conjecture for $\SL_n(\eps q)$, which generalises \cite[Thm.~1.2]{Feng19}.

\begin{mainthm}\label{awc-slsu}
The blockwise Alperin weight conjecture  holds for all finite special linear and unitary groups and all primes.
\end{mainthm}

The (iBAW) condition is verified for the unipotent blocks of $\SL_n(\eps q)$ for $\ell\nmid \mathrm{gcd}(n,q-\eta)$ in \cite{Feng19}.
The following result generalized this work to all primes and more blocks.

\begin{mainthm}\label{uni-maxdef}
	Let $G=\SL_n(\eps q)$, $\ell$ a prime not dividing $q$ and $B$ an $\ell$-block of $G$.
	If
	$B$ is either a unipotent block
	or a block of maximal defect,
	then the inductive blockwise Alperin weight  condition holds for  $B$.  
\end{mainthm}

For general blocks of special linear and unitary groups, we prove the following theorem which 
generalized the work for simple groups of type $\mathsf A$ with trivial Schur multiplier
in \cite{LZ18,LZ19}.

\begin{mainthm}\label{ibaw-nq-1-prime}
Let $S=\PSL_n(\eps q)$ be a simple group
and $\ell$ be a prime not dividing $q$. 	
	Assume that $\gcd(n,q-\eps)_{\ell'}$ is square-free.
	Then the inductive blockwise Alperin weight condition holds for $S$ and $\ell$.
\end{mainthm}

From this, we prove that the (iBAW) condition holds for simple groups $\PSL_n(q)$ and $\PSU_n(q)$ for $n\le 7$ (see Proposition \ref{iBAW-small-rank}).

In \cite{Sp13}, Sp\"ath proved the blockwise Alperin weight conjecture  holds for finite groups with abelian Sylow $2$-subgroups  via the (iBAW) condition.
We will consider finite groups with abelian Sylow $3$-subgroups and
prove in this paper the following result.

\begin{mainthm}\label{abel-syl-3}
	Assume that G is a finite group with abelian Sylow 3-subgroups.
Then the blockwise Alperin weight conjecture holds for $G$ and any prime.
\end{mainthm}

The last section of this paper is a continuation of \cite[\S 4]{Feng19}.
There is a classification of the $\ell$-blocks of $\SL_n(\eta q)$ in \cite[\S 4]{Feng19} when $\ell\nmid q$ is odd, using the labelling set of $d$-Jordan-cuspidal pairs given by Cabanes--Enguehard \cite{CE99} and Kessar--Malle~\cite{KM15}.
The results related to $d$-Jordan-cuspidal pairs in  \cite{CE99,KM15} are restricted to odd primes.
In Section \ref{2-blocks-SLSU} of this paper, we consider the 2-blocks of $\SL_n(\eta q)$ (with odd $q$). 
Our result relies on the classification of the Brauer pairs of $\GL_n(\eta q)$ by Brou\'e~\cite{Brou86} and the description of the radical subgroups of $\SL_n(\eta q)$ in~\cite{FLZ20}.
The number of 2-blocks of $\SL_n(\eta q)$  covered by a given 2-blocks of $\GL_n(\eta q)$ is determined in Remark \ref{blocksofslsu},
and in this way we obtain a parametrization for the 2-blocks of finite special linear and unitary groups, which complements the result of \cite[\S 4]{Feng19}.

\vspace{2ex}
We begin Section \ref{criterion-iBAW} by recalling the criterion for the (iBAW) condition given by Brough-Sp\"ath and give some notation and preliminaries for linear and unitary groups in Section \ref{notation-linear-unitary-groups}.
Using the results of \cite{FLZ20}, it suffices to consider only one condition for type $\mathsf A$, and from this we verify the (iBAW) condition for certain cases of groups of type $\mathsf A$ and prove Theorem \ref{awc-slsu}--\ref{ibaw-nq-1-prime} in Section \ref{sec:iBAW-A}.
The blockwise Alperin weight conjecture is verified to hold for finite groups with abelian Sylow $3$-subgroups in Section \ref{AWC-abel-sylow3}. 
In Section \ref{2-blocks-SLSU}, we give a classification of the 2-blocks of finite special linear and unitary groups; see Remark \ref{blocksofslsu}.


\section{A criterion for the inductive blockwise Alperin weight condition}\label{criterion-iBAW}

All groups considered in this paper are finite.
For the notation for the block and character theory, we mainly follow \cite{Is76,Na98}, except that
we denote the restriction of $\chi\in\Irr(G)\cup\IBr_\ell(G)$ to some subgroup $H\le G$ by $\Res^G_H\chi$, while $\Ind^G_H\psi$ denotes the character induced from $\psi\in\Irr(H)\cup\IBr_\ell(H)$ to $G$.

If a group $A$ acts on a finite set $X$, we denote by $A_x$ the stabilizer of $x\in X$ in $A$, analogously we denote by $A_{X'}$ the setwise stabilizer of $X'\subseteq X$.
If $A$ acts on a group $G$ by automorphisms, there is a natural action of $A$ on $\Irr(G)\cup\IBr_\ell(G)$ given by ${}^{a^{-1}}\chi(g)=\chi^a(g)=\chi(g^{a^{-1}})$ for every $g\in G$, $a\in A$ and $\chi\in\Irr(G)\cup\IBr_\ell(G)$.
For $P\le G$ and $\chi\in\Irr(G)\cup\IBr_\ell(G)$, we denote by $A_{P,\chi}$ the stabilizer of $\chi$ in $A_P$.
For $N\unlhd G$ we sometimes identify the characters of $G/N$ with the characters of $G$ whose kernel contains $N$.
If $\chi\in\Irr(G)$, then $\chi^0$ is used for the restriction to the $\ell$-regular elements of
$G$.

If $G$ is abelian, we also write $\Lin(G)=\Irr(G)$ since all irreducible characters of $G$ are linear.
Let $\Lin_{\ell'}(G)$ denote the element of $\Lin(G)$ of $\ell'$-order.
Then the map $\Lin_{\ell'}(G)\to \IBr_\ell(G)$, $\chi\mapsto \chi^0$ is bijective.
From this, we always identify $\IBr_\ell(G)$ with $\Lin_{\ell'}(G)$ when $G$ is abelian.

For a finite group $G$, we denote by $\dz_\ell(G)$ the set of all $\ell$-defect zero characters of $G$.
If $\chi\in\Irr(G)\cup\IBr_\ell(G)$, we write $\bl_\ell(\chi)$ for the $\ell$-block of $G$ containing $\chi$.
If $R$ is a radical $\ell$-subgroup of $G$ and $B$ is an $\ell$-block of $G$, then we define the set
$$\dz_\ell(N_G(R)/R,B):=\{\varphi\in\dz_\ell(N_G(R)/R)\mid\bl_\ell(\varphi)^G=B\},$$
where we regard $\varphi$ as an irreducible character of $N_G(R)$ containing $R$ in its kernel when considering the induced block $\bl_\ell(\varphi)^G$.

\vspace{2ex}

The inductive blockwise Alperin weight (iBAW)  condition can be stated using the notion of modular character triples  and isomorphisms between them (for background on modular character triples, see, e.g., \cite[\S 8]{Na98}).
By \cite[Thm.~4.4]{Sp17},
an $\ell$-block $B$ of $G$ satisfies 
the  (iBAW)  condition from \cite[Def.~4.1]{Sp13}
if for $\Gamma:=\Aut(G)_B$
there exists a $\Gamma$-equivariant bijection $\Omega:\IBr_\ell(B)\to\Alp_\ell(B)$ such that for every $\psi\in\IBr_\ell(B)$ and $\Omega(\psi)=(R,\varphi)$, one has
\begin{equation}\label{block-iso-triple}
(G\rtimes \Gamma_{R,\varphi}, G,\psi )\geqslant_b(N_G(Q) \rtimes \Gamma_{R,\varphi},N_G(Q),\varphi^0).
\addtocounter{thm}{1}\tag{\thethm}
\end{equation}
For the definition of the relation  $\geqslant_b$, which is called the \emph{block isomorphism of modular character triples},
see \cite[Def.~3.2]{Sp17}.

Recently,  J. Brough and B. Sp\"ath  \cite[Thm.~4.5]{BS19} gave a 
criterion for the inductive Alperin weight condition adapted to quasi-simple groups of Lie type with abelian outer automorphism groups.
Here we rewrite it and give a new version suitable for quasi-simple groups with possibly non-abelian outer automorphism groups.
In fact, conditions (i) -- (iv) are the same, and condition
(v) which considers relations of blocks for irreducible constituents of (Brauer) characters
is altered.

\begin{thm}\label{thm:criterion-block}
	Let $S$ be a finite non-abelian simple group and $\ell$ a prime dividing $|S|$.
	Let $G$ be the universal $\ell'$-covering group of $S$, $\mathcal B$ a union of $\ell$-blocks of $G$ and assume
 there are groups $\tG$, $D$ such that $G \unlhd \tG \rtimes D$,
$\mathcal B$ is a $\tG$-orbit
and the following hold.
	\begin{enumerate}[\rm(i)]\setlength{\itemsep}{0pt}
		\item \begin{enumerate}[\rm(a)]\setlength{\itemsep}{0pt}
			\item $G=[\tG,\tG]$ and $D$ is abelian,
			\item $C_{\tG D}(G)=Z(\tG)$ and $\tG D/Z(\tG) \cong \Aut(G)$,
			\item any element of $\IBr_\ell(\mathcal B)$ extends to its stabilizer in $\tG$,
			\item for any radical $\ell$-subgroup $R$ of $G$ and any $B\in \mathcal{ B}$, any element of $\dz_\ell(N_G(R)/R\mid B)$ extends to its stabilizer in $N_{\tG}(R)/R$.
		\end{enumerate}
		\item Let $\mathcal{\tilde B}$ be the union of $\ell$-blocks of $\tG$ covering $\mathcal B$.
There exists a $\Lin_{\ell'}(\tG/G) \rtimes D_{\mathcal{\tB}}$-equivariant bijection
		$\tOmega: \IBr_\ell(\mathcal{\tilde B}) \to \Alp_\ell(\mathcal{\tilde B})$ such that
		\begin{enumerate}[\rm(a)]\setlength{\itemsep}{0pt}
			\item $\tOmega(\IBr_\ell(\tilde B)) =  \Alp_\ell(\tilde B)$
			for every $\tilde B\in\tcB$,
			\item $\J_G(\tpsi) = \J_G(\tOmega(\tpsi))$ for every $\tpsi \in \IBr_\ell(\tG)$.
		\end{enumerate}
		\item For every $\tpsi\in\IBr_\ell(\mathcal{\tilde B})$, there exists some $\psi_0\in\IBr_\ell(G\mid\tpsi)$ such that
		\begin{enumerate}[\rm(a)]\setlength{\itemsep}{0pt}
			\item $(\tG\rtimes D)_{\psi_0}=\tG_{\psi_0}\rtimes D_{\psi_0}$,
			\item $\psi_0$ extends to $G\rtimes D_{\psi_0}$.
		\end{enumerate}
		\item For every $(\tR,\tvarphi)\in\Alp_\ell(\mathcal{\tilde B})$, there is an $\ell$-weight $(R,\varphi_0)$ of $G$ covered by $(\tR,\tvarphi)$ such that
		\begin{enumerate}[\rm(a)]\setlength{\itemsep}{0pt}
			\item $(\tG D)_{R,\varphi_0} = \tG_{R,\varphi_0} (GD)_{R,\varphi_0}$,
			\item $\varphi_0$ extends to $(G\rtimes D)_{R,\varphi_0}$.
		\end{enumerate}
\item If the $\ell$-Brauer character $\tpsi$ in (iii) and the weight $(\tR,\tvarphi)$ in (iv) satisfy $\overline{(\tR,\tvarphi)}=\tilde\Omega(\tpsi)$,
then the $\ell$-Brauer character $\psi_0$ and the weight $(R,\varphi_0)$  can be chosen in the same block $B\in\mathcal B$ and
satisfy 
\begin{equation}\label{equ:block-corr}
	\bl_\ell(\hat\psi)=\bl_\ell(\hat\varphi)^{\tG_\psi}
	\addtocounter{thm}{1}\tag{\thethm}
\end{equation} 
where $\hat\psi\in\IBr_\ell(\tG_\psi\mid\psi)$ is the Clifford correspondent of $\tpsi$, $\hat\varphi$ is an extension of $\varphi_0$ to $N_{\tG}(R)_\varphi/R$ such that via induction and the map $\Delta_\varphi$ from \cite[Thm.~2.10]{BS19} it corresponds to $(\tR,\tvarphi)$.

	\end{enumerate}
	Then the inductive blockwise Alperin weight  (iBAW)  condition holds for every block $B\in\mathcal B$.
\end{thm}

For the definition of $\J_G(\tilde\psi)$, see \cite[\S 2]{BS19}.

\begin{proof}
We use the construction in
\cite[Thm.~4.5]{BS19}.
In fact, according to its proof, 
the conditions (i)--(iv) already give a $(\tG \rtimes D)_{\mathcal B}$-equivariant bijection $\Omega:\IBr_\ell(\mathcal B)\to\Alp_\ell(\mathcal B)$.
Similar as the proof of \cite[Lemma~4.6]{BS19}, we can show that (\ref{block-iso-triple}) holds via $\Omega$.
Thus it suffices to show that $\Omega$ preserves blocks.
For $\tpsi\in\IBr_\ell(\tilde{\mathcal B})$ and  weight $\overline{(\tR,\tvarphi)}\in\Alp_\ell(\tilde{\mathcal B})$ satisfying $\overline{(\tR,\tvarphi)}=\tilde\Omega(\tpsi)$,
we let $\psi_0$ and $(R,\varphi_0)$ satisfy (iii) and (iv) respectively.
Then 
$\psi_0$ and $\overline{(R,\varphi_0)}$ have the same stabilizer in $\tG\rtimes D$
and by the construction of $\Omega$ in \cite{BS19}, we have
$\Omega(\psi_0)=\overline{(R,\varphi_0)}$.
From this, if $\psi_0$ and $(R,\varphi_0)$ can be chosen in the same block of $G$, then $\Omega$ preserves blocks, as desired.
\end{proof}

Let $S$ be a non-abelian finite simple group, $\ell$ a prime dividing $|S|$ and $G$ the universal $\ell'$-covering group of $S$. We say that \emph{the (iBAW) condition holds for $S$ and $\ell$} if the (iBAW) condition holds for every $\ell$-block of $G$.
Moreover, we say \emph{the (iBAW) condition holds for $S$} if the  (iBAW) condition holds for $S$ and any prime $\ell$ dividing $|S|$.

\begin{lem}\label{lem:covering}
	Let $\tG$ be an arbitrary finite group and $G\unlhd \tG$ with abelian quotient $\tG/G$.
	Let $\tpsi\in\IBr_\ell(\tG)$, $\psi\in\IBr_\ell(G\mid\tpsi)$, $\tB=\bl_\ell(\tpsi)$ and $B=\bl_\ell(\psi)$.
	Assume that $\psi$ extends to $\tG_\psi$ and  $\gcd(|\tG_\psi/GZ(\tG)|_{\ell'},|\tG/\tG_\psi|)=1$.
	Then there is a unique block $\hat B$ of $\tG_\psi$ such that $\tB$ covers $\hat B$ and $\hat B$ covers $B$.
\end{lem}	

\begin{proof}
	Note that 
$GZ(\tG)\cong G\times Z(\tG)/(G\cap Z(\tG))$.
Let $\IBr_\ell(G\cap Z(\tG)\mid\psi)=\{\tau\}$.
Then the blocks of $GZ(\tG)$ covering $B$ are parametrized by $\IBr_\ell(Z(\tG)\mid\tau)$.
More precisely, those blocks are $B'_\la$ with $\la\in \IBr_\ell(Z(\tG)\mid\tau)$ and $\IBr_\ell(B'_\la)=\{\, \la\cdot\chi\mid \chi\in\IBr_\ell(B) \,\}$ where $\la\cdot\chi(zg)=\la(z)\chi(g)$ for the $\ell$-regular elements $z\in Z(\tG)$ and $g\in G$.
Any two of those blocks are not $\tG$-conjugate.
For every subgroup $GZ(\tG)\le G_1\le\tG$ and every block $B_1$ of $G_1$, 
we let $\la_1\in \IBr_\ell(Z(\tG)\mid\chi_1)$ for $\chi_1\in\IBr_\ell(B_1)$.
Then
$B_1$ covers $B$ if and only if $B_1$ covers $B'_{\la_1}$.
So we can assume that $G=GZ(\tG)$ without loss of generality.

	Let $\hat G=\tG_\psi$ and let $\hat\psi\in\IBr_\ell(\hat G)$ be the Clifford correspondent of  $\tpsi$.
	Then $\hat \psi$ is an extension of $\psi$.
	Let $\hat B_0=\bl_\ell(\hat\psi)$.
	Then $\tB$ covers $\hat B_0$ and $\hat B_0$ covers $B$.
	Let $\mathcal L=\{\, \hat\psi^g\mid g\in\tG \,\}$ and $\mathcal S=\{\, \lambda\hat\psi\mid\lambda\in\IBr_\ell(\hat G/G) \,\}$.
	Then $\mathcal L=\IBr_\ell(\hat G\mid\tpsi)$ and $\mathcal S=\IBr_\ell(\hat G\mid\psi)$.
	Moreover, $\mathcal S$ consists of the extensions of $\psi$ to $\hat G$ and $\mathcal L\cap \mathcal S=\{\hat\psi\}$.
	
	Let $\mathcal A=\{\, \lambda\hat\psi^g\mid g\in\tG,\lambda\in\IBr_\ell(\hat G/G) \,\}$ be the $\IBr_\ell(\hat G/G)\times (\tG/\hat G)$-orbit on $\IBr_\ell(\hat G)$ containing $\hat\psi$.
	Note that $\IBr_\ell(\hat G/G)$ permutes the blocks of $\hat G$ covering $B$.
	Now  $\gcd(|\tG_\psi/G|_{\ell'},|\tG/\tG_\psi|)=1$, so every subgroup of $\IBr_\ell(\hat G/G)\times (\tG/\hat G)$ is of form $H\times K$ with $H\le \IBr_\ell(\hat G/G)$ and $K\le \tG/\hat G$.
In particular,		the stabilizer of $\hat B_0$ in $\IBr_\ell(\hat G/G)\times (\tG/\hat G)$ is $\mathcal I\times (\tG_{\hat B_0}/\hat G)$
where $\mathcal I$ is the stabilizer of $\hat B_0$ in $\IBr_\ell(\hat G/G)$.
		Thus $\IBr_\ell(\hat B_0)\cap\mathcal A=\{\, \lambda\hat\psi^g\mid g\in\tG_{\hat B_0}/\hat G,\lambda\in\mathcal I  \,\}$.
	
	Assume that $\hat B$ is a block of $\hat G$ such that  $\tB$ covers $\hat B$ and $\hat B$ covers $B$. 
	Then there exists $g\in\tG$ such that $\hat B=\hat B_0^g$. In particular, $\hat\psi^g\in \IBr_\ell(\hat B)$.
	Since $\hat B$ covers $B$, one has that $\hat B=\lambda\otimes\hat B_0$ with $\lambda\in\IBr_\ell(\hat G/G)$.
	Then $\hat\psi^g=\lambda \hat \psi'$ with $\hat\psi'\in\IBr_\ell(\hat B_0)$.
	It follows that $\hat \psi'=\lambda^{-1}\hat\psi^g\in\mathcal A$.
	Let $\hat\psi'=\lambda'\hat\psi^{g'}$ with $\lambda'\in\mathcal I$ and $g'\in \tG_{\hat B_0}$.
	Then $\hat\psi^{g{g'}^{-1}}=\lambda\lambda'\hat\psi$.
	This implies that  $\lambda\lambda'\in\mathcal I_{\hat\psi}$ and $g{g'}^{-1}\in\tG_{\hat\psi}$ because $\mathcal L\cap \mathcal S=\{\hat\psi\}$.
	So $g\in\tG_{\hat B_0}$ and $\hat B=\hat B_0$, which complete the proof.
\end{proof}

\begin{prop}\label{prop:(2.3)}
In Theorem~\ref{thm:criterion-block}, if moreover one of the following holds:
\begin{enumerate}[\rm(i)]
		\item if $B'$ is a block of some group $J$ with $G\le J\le \tG$ that is covered by a block of $\tcB$, then $\tG_{B'}=\tG$,
	\item  there is some  $\chi\in\Irr(\tcB)\cup\IBr_\ell(\tcB)$ such that $\Res^{\tG}_{G}\chi$ is irreducible,
	\item $|\tG/GZ(\tG)|_{\ell'}$ is square-free,
	\end{enumerate}
then (\ref{equ:block-corr}) holds.
\end{prop}

\begin{proof}
Let $\psi_0$ and $(R,\varphi_0)$ be as in Theorem~\ref{thm:criterion-block} (v).
We will prove that
$\bl_\ell(\hat\psi)=\bl_\ell(\hat\varphi)^{\tG_\psi}$,
where $\hat\psi\in\IBr_\ell(\tG_\psi\mid\psi)$ is the Clifford correspondent of $\tpsi$, and $\hat\varphi$ is an extension of $\varphi_0$ to $N_{\tG}(R)_\varphi/R$ such that via induction and the map $\Delta_\varphi$ (from \cite[Thm.~2.10]{BS19}) it corresponds to $(\tR,\tvarphi)$.
Since $\Omega$ and $\tOmega$ respect blocks, by 
\cite[Lemma~2.3]{KS15}, one has $\bl_\ell(\hat\varphi)^{\tG_\psi}$ covers $\bl_\ell(\psi)$. Also by \cite[Thm.~2.10]{BS19}, $\bl_\ell(\tpsi)$ covers $\bl_\ell(\hat\varphi)^{\tG_\psi}$.
For (i) or (iii), $\bl_\ell(\tpsi)$ covers a unique block of $\tG_\psi$.
For (ii),  Lemma~\ref{lem:covering} also implies $\bl_\ell(\hat\psi)=\bl_\ell(\hat\varphi)^{\tG_\psi}$.
This completes the proof.
\end{proof}

\section{Linear and unitary groups}\label{notation-linear-unitary-groups}

We will follow the notation in \cite{FLZ20} for linear and unitary groups.
Much of this notation originally comes from \cite{AF90, An92,An93, An94,FS82}.

Assume $q=p^f$ is a power of a prime $p$ and $n\ge 2$.
Let $\GL_n(q)$ be the  the group of all
invertible $n\times n$ matrices over $\mathbb F_q$.
Also we denote by $F_p$, $F_q$ and $\sigma_{it}$
the field automorphism, standard Frobenius endomorphism and graph automorphism respectively;
see the definitions for example in \cite[\S 2]{FLZ20}.
Recall that $\GL_n(-q)$ denotes the general unitary group
$$\GU_n(q)=\{~ A\in\GL_n(q^2) ~|~ (F_q(A))^tA=I_n ~\},$$
where $I_n$ is the identity matrix of degree $n$.
We will use the similar notation $\SL_n(-q)$ ($\PSL_n(-q)$) for  $\SU_n(q)$ ($\PSU_n(q)$). 
Let $\tG=\GL_n(\eta q)$ and
$G=\SL_n(\eta q)$ for $\eps=\pm 1$.
We define $D=\langle F_p,\sigma_{it}  \rangle$ if $n\ge 3$ while $D=\langle F_p  \rangle$ if $n=2$.
Then by~\cite[Thm. 2.5.1]{GLS98}, $(\tG\rtimes D)/Z(\tG)\cong \Aut(G)$.

We also recall the subset $\cF$ of the set of monic irreducible polynomials from~\cite[\S 1]{FS82}. 
Denote by $\Irr(\F_q[X])$ the set of all monic irreducible polynomials over the field $\F_q$.
For $\Delta(X)=X^m+a_{m-1}X^{m-1}+\cdots+a_0$ in $\F_{q^2}[X]$, we define $\tDelta(X)=X^ma_0^{-q}\Delta^q(X^{-1})$, where $\Delta^q(X)$ means the polynomial in $X$ whose coefficients are the $q$-th powers of the corresponding coefficients of $\Delta(X)$.
Now, we denote by
\begin{align*}
\cF_0 &= \left\{~ \Delta ~|~ \Delta\in\Irr(\F_q[X]),\Delta\neq X ~\right\},\\
\cF_1 &= \left\{~ \Delta ~|~ \Delta\in\Irr(\F_{q^2}[X]),\Delta\neq X,\Delta=\tDelta ~\right\},\\
\cF_2 &= \left\{~ \Delta\tDelta ~|~ \Delta\in\Irr(\F_{q^2}[X]),\Delta\neq X,\Delta\neq\tDelta ~\right\}
\end{align*}
and
$$\cF=\left\{ \begin{array}{ll} \cF_0 & \textrm{if}~\eta=1;\\ \cF_1\cup\cF_2 & \textrm{if}~\eta=-1. \end{array} \right.$$
We denote by $d_\Gamma$ the degree of any polynomial $\Gamma$.
For any semisimple element $s$ of $\tG$, we let $s=\prod_\Gamma s_\Gamma$ be its primary decomposition.
We denote by $m_\Gamma(s)$ the multiplicity of $\Gamma$ in $s_\Gamma$.
If $m_\Gamma(s)$  is not zero, we call $\Gamma$ an \emph{elementary divisor}.
Denote by $\cF'$ the subset of $\cF$ of those polynomials whose roots are of $\ell'$-order.
For $\Gamma \in \cF$, denote by  $e_\Gamma$ the multiplicative order of $(\eps q)^{d_\Gamma}$ modulo $\ell$. Also, we define $e$ to be the multiplicative order of $\eps q$ modulo $\ell$. Note that $e=e_\Gamma=1$ when $\ell=2$. 

Let $F_{\eta{q}}=\sigma_{it}^{\frac{1-\eta}{2}}F_q$.
Then $F_{\eta q}$ acts on $\barF_q^\times$ by $F_{\eta q} (\xi) = \xi^{\eta q}$.
A polynomial $\Gamma\in\cF$ can be identified with the set of roots of $\Gamma$, which can be again identified with an $F_{\eta q}$-orbit $\group{F_{\eta{q}}}\cdot\xi$ of this action, where $\xi$ is a root of $\Gamma$; see for example \cite[\S3.1]{De17}.
Let $\mathfrak Z=\{z\in \barF_q^\times\mid z^{q-\eta}=1\}$.\label{def-frakZ}
For any $z\in\mathfrak Z$ and $\Gamma\in\cF$, $z.\Gamma$ is defined to be the polynomial in $\cF$ whose roots are the roots of $\Gamma$ multiplied by $z$, defining an action of $\mathfrak Z$ on $\cF$.
Note that we can identify $Z(\tG)$ with $\mathfrak Z$.

\paragraph{}
For the representations of finite groups of Lie type, see for example \cite{CE04}.
Let $\tbG=\GL_n(\barF_p)$, $F=F_{\eta q}$, then $\tG=\tbG^F$.
If $\tbL$ is a Levi subgroup of a reductive group $\tbG$ with the Frobenius map $F$,
then by the fact that $Z(\tbG)$ is connected,
there is an isomorphism (see for example \cite[(8.19)]{CE04})
\begin{equation*}\label{linear-char-levi}
Z(\tbL)^F \to \Irr(\tbL^F/[\tbL^F,\tbL^F]),\quad z\mapsto\hz.
\end{equation*}
If $s$ is a semisimple element of $\tG$, then $C_{\tbG}(s)$ is a Levi subgroup of $\tbG$.

\section{The inductive blockwise Alperin weight conditions for simple groups of type $\mathsf A$}\label{sec:iBAW-A}

Given a semisimple element $s$ of $\tG=\GL_n(\eps q)$, let $\prod\limits_{\Gamma\in \mathcal{F}} s_{\Gamma}$ be the primary decomposition of $s$.
Here $s_\Gamma$ is conjugate to $m_\Gamma(s)(\Gamma)$, where $(\Gamma)$ is the  companion matrix of $\Gamma$.
Thus $n=\sum\limits_{\Gamma\in \mathcal{F}}m_\Gamma(s)d_\Gamma$.
Jordan decomposition gives a bijection between the irreducible characters of $\tG=\GL_n(\eps q)$ and the $\tG$-conjugacy classes of pairs $(s,\mu)$,
where $s=\prod\limits_{\Gamma\in\cF} m_\Gamma(s)(\Gamma)$ is a semisimple element of $\tG$ and $\mu=\prod\limits_{\Gamma\in\cF} \mu_\Gamma$ with $\mu_\Gamma\vdash m_\Gamma(s)$.
See for instance \cite[\S 1]{FS82} and \cite[Chap.~8]{CE04}.

The blocks of $\tG=\GL_n(\eps q)$ have been classified in \cite{FS82,Brou86}: the $\ell$-blocks of $\tG$ are in bijection with the set of $\tG$-conjugacy classes of pairs $(s,\lambda)$, where $s$ is a semisimple $\ell'$-element of $\tG$ and $\lambda=\prod_\Gamma \lambda_\Gamma$ with $\lambda_\Gamma$ the $e_\Gamma$-core of a partition of $m_\Gamma(s)$. Recall that $e_\Gamma$ is the multiplicative order of $(\eps q)^{d_\Gamma}$ modulo $\ell$.
Note that, for $\ell=2$, $(s,\lambda)$ is always $(s,-)$ (here, $-$ denotes the empty partition), which means that $\mathcal E_2(\tG,s)$ is a single 2-block of $\tG$.
Let $\tB$ be an $\ell$-block of $\tG$ corresponding to $(s,\lambda)$. 
Then an irreducible $\ell$-Brauer character with labeling $(s',\lambda')\in i \IBr(\tG)$ (see~\cite{LZ18} for the notation) is in the block $\tB$ if and only if  $s'$ is $\tG$-conjugate to $s$ and $\lambda'_\Gamma$ has $e_\Gamma$-core $\lambda_\Gamma$ for every $\Gamma$.

On the other hand, the $\tB$-weights  are classified in \cite{AF90, An92, An93, An94} and we will use the explicit labelling $(s,\lambda,K)\in i\Alp(\tG)$ from 
 \cite[\S 3]{FLZ20}.
 In addition, an $\ell$-weight with labeling $(s',\lambda',K')$ is in the block $\tB$ with label $(s,\lambda)$ if and only if 
$(s',\lambda')$ is $\tG$-conjugate to $(s,\lambda)$.

The first consequence of the results in \cite{FLZ20} is that
Theorem \ref{awc-slsu} holds.

\begin{proof}[Proof of Theorem \ref{awc-slsu}]
	Thanks to \cite{Ca88},
	we only need to consider the non-defining characteristic.	
	For $G=\SL_n(\eps q)$ and $\tG=\GL_n(\eps q)$, 
	we let $\tilde B$ be an $\ell$-block of $\tG$ and $\mathcal B$ the union of $\ell$-blocks of $G$ covered by $\tB$.
	By the correspondence between $\IBr_\ell(\tilde B)$ and $\Alp_\ell(\tilde B)$ in \cite{AF90,An92,An93,An94}, the proof of the main theorem of \cite{FLZ20}
	indeed obtained that $|\IBr_\ell(\mathcal B)|=|\Alp_\ell(\mathcal B)|$, which implies that $|\IBr_\ell(B)|=|\Alp_\ell(B)|$ for every $B\in\mathcal B$
	 immediately since the blocks in $\mathcal B$ are $\tG$-conjugate.
\end{proof}

Now we consider the (iBAW) condition for the blocks of $G=\SL_n(\eta q)$.

\begin{thm}\label{iBAW-typeA}
Assume that $G=\SL_n(\eps q)$ is the universal covering of the finite simple group $S=\PSL_n(\eps q)$.
	Let $\ell$ be a prime not dividing $q$, $\tG=\GL_n(\eps q)$ and $\mathcal B$ be a $\tG$-orbit of $\ell$-blocks of $G$.
If furthermore the condition (v) of  Theorem \ref{thm:criterion-block} holds for the bijection $\tilde\Omega$ in \cite[\S 6]{FLZ20},  then the  (iBAW)  condition holds for any $B\in \mathcal B$.
\end{thm}

\begin{proof}
By the above observations, the bijection $\tilde\Omega$ used in \cite{FLZ20} perserves blocks,
i.e., condition (ii)(c) of Theorem \ref{thm:criterion-block} holds.
Thus by the proof of the main theorem of \cite{FLZ20},
conditions (i)--(iv)  of Theorem \ref{thm:criterion-block} hold.
From this, 
if condition (v) of  Theorem \ref{thm:criterion-block} holds,
then the  (iBAW)  condition holds for any $B\in \mathcal B$.
\end{proof}

\begin{rem}
	Keep the notation of Theorem  \ref{iBAW-typeA}, we know from its proof  that if (\ref{equ:block-corr}) holds and there is a
$(\tG \rtimes D)_{\mathcal B}$-equivariant bijection between $\IBr_\ell(\mathcal B)$ and $\Alp_\ell(\mathcal B)$ which preserves blocks, then the   (iBAW)  condition from holds for every block $B\in\mathcal B$.
\end{rem}

Now we consider the  (iBAW)  condition for certain blocks of groups of 
type $\mathsf A$.

\begin{thm}\label{iBAW-SLSU-some-special-case}
Let $G=\SL_n(\eps q)$, $\tG=\GL_n(\eps q)$ and $\ell$ a prime not dividing $q$.
Let $\mathcal B$ be a $\tG$-orbit of $\ell$-blocks of $G$.
If (\ref{equ:block-corr}) holds,
then the  (iBAW)  condition holds for every $B\in\mathcal B$ if one of the following is satisfied.
\begin{enumerate}[\rm(i)]
\item $\tG_B=\tG$ for $B\in\mathcal B$.
\item If $\psi\in \IBr_\ell(\mathcal B)$ satisfies that
$(\tG\rtimes D)_\psi=\tG_\psi\rtimes D_\psi$, then for any $B\in \mathcal B$ and any
$g\in \tG\setminus \tG_B$, 
there exists $g_0\in g \tG_B$ such that either $[\langle \tG_\psi,g_0 \rangle, D_\psi]\subseteq \tG_\psi$ or $\tG_\psi [\langle \tG_\psi,g_0 \rangle, D_\psi]\nsubseteq\tG_B$.
\item If $\psi\in \IBr_\ell(\mathcal B)$ satisfies that
$(\tG\rtimes D)_\psi=\tG_\psi\rtimes D_\psi$, then for any $B\in \mathcal B$ and any
$g\in \tG\setminus \tG_B$, 
there exists $g_0\in g \tG_B$ such that $\langle \tG_\psi,g_0 \rangle\cap\tG_B=\tG_\psi$.
\item $\tG(\tG\rtimes D)_B/GZ(\tG)$ is abelian.
\item $\gcd(|\tG:\tG_B|,|\tG_B:\tG_\psi|)=1$ for any $B\in \mathcal B$ and $\psi\in\IBr_\ell(\mathcal B)$.
\item $\tG_B=\tG_\psi$, for any $B\in \mathcal B$ and $\psi\in\IBr_\ell(\mathcal B)$.
\item $\tG_B=GZ(\tG)$ for $B\in\mathcal B$.
\end{enumerate}
\end{thm}

\begin{proof}
We use Theorem \ref{iBAW-typeA}, and then (i) follows immediately.
For (ii),
we assume that $\tG_B<\tG$ for $B\in\mathcal B$.
Let $\psi$, $(R,\varphi)$ satisfy condition (iii) and (iv) of Theorem \ref{thm:criterion-block}
respectively and let $B\in\mathcal B$ and $g\in \tG$ such that
$\psi\in\IBr_\ell(B)$ and $\overline{(R,\varphi)}\in\Alp_\ell(B^g)$.
If $B^g=B$, then the assertion holds by Theorem \ref{iBAW-typeA}.
Thus we assume that $g\notin\tG_B$.
We claim that $\tG_{\psi} [\langle \tG_{\psi},g' \rangle, D_{\psi}]\subseteq\tG_B$ for any $g'\in g\tG_B$.
Obviously, $B^g=B^{g'}$.
By the construction in  \cite{BS19}, $(\tG\rtimes D)_{\psi}=(\tG\rtimes D)_{\overline{(R,\varphi)}}$.
Since $\overline{(R,\varphi)}\in\Alp_\ell(B^g)$,
we know $\tG_{\psi}\rtimes D_{\psi} = (\tG\rtimes D)_{\overline{(R,\varphi)}} \le (\tG\rtimes D)_{B^{g}}$.
Since $\tG/G$ is cyclic, any subgroup of $\tG$ containing $G$ is normal in $\tG$.
Also  $(\tG\rtimes D)_{\psi^{g'}}=\tG_{\psi}D_{\psi}^{g'}\le(\tG\rtimes D)_{B^{g'}}$.
For $\sigma\in D_{\psi}$, we have $\sigma^{g'}=\sigma [\sigma, g']$.
Thus $[\langle \tG_{\psi},g' \rangle, D_{\psi}]\subseteq (\tG\rtimes D)_{B^{g'}}$ because $D_{\psi}\subseteq (\tG\rtimes D)_{B^{g'}}$.
So $\tG_{\psi} [\langle \tG_{\psi},g' \rangle, D_{\psi}]\subseteq\tG_{B^{g'}}=\tG_B$ and the claim holds.

Therefore, $\tG_{\psi} [\langle \tG_{\psi},g_0 \rangle, D_{\psi}]\subseteq\tG_{\psi}$ for some $g_0\in g \tG_B$
by the hypothesis.
Then $(\tG\rtimes D)_{\psi^{g_0}}=\tG_{\psi} D_{\psi}^{g_0}\le \tG_{\psi} D_{\psi}$.
From this we have  $(\tG\rtimes D)_{\psi^{g_0}}=\tG_{\psi^{g_0}}\rtimes D_{\psi^{g_0}}$.
Thus both the Brauer character $\psi^{g_0}$ and the weight $(R,\varphi)$ lie in the block $B^{g_0}=B^g$ and  satisfy conditions (iii) and (iv) of Theorem \ref{thm:criterion-block}.
So the assertion follows by Theorem \ref{iBAW-typeA}.

Now we consider (iii). 
Note that $[\langle \tG_{\psi},g' \rangle, D_{\psi}]\le\langle \tG_{\psi},g' \rangle$.
By the hypothesis, for any $B\in \mathcal B$ and any
$g\in \tG\setminus \tG_B$, 
there exists $g_0\in g \tG_B$ such that $\langle \tG_\psi,g_0 \rangle\cap\tG_B=\tG_\psi$.
So either $[\langle \tG_\psi,g_0 \rangle, D_\psi]\subseteq \tG_\psi$ or $\tG_\psi [\langle \tG_\psi,g_0 \rangle, D_\psi]\nsubseteq\tG_B$.
For the rest, we mention that the implications (ii) $\Rightarrow$ (iv) and
(iii) $\Rightarrow$ (v) $\Rightarrow$ (vi) $\Rightarrow$ (vii) are direct
and we complete the proof.
\end{proof}

\begin{proof}[Proof of Theorem \ref{uni-maxdef}]
When $B$ is a unipotent block, we have $\tG_B=\tG$ by  \cite[Remark~4.13]{Feng19}, where $\tG=\GL_n(\eps q)$.
Thus the assertion follows from Proposition~\ref{prop:(2.3)} (ii) and
Theorem~\ref{iBAW-SLSU-some-special-case}~(i).

If $B$ is of maximal defect, then by \cite[Prop.~5.4]{CS15}, one has that Proposition~\ref{prop:(2.3)} (i) is satisfied. So the assertion follows from Theorem~\ref{iBAW-SLSU-some-special-case}~(i).
\end{proof}

Let $S\in\{\PSL_n(q),\PSU_n(q)\}$ be a simple group of type $\mathsf A$ and $G$ be the universal covering group of $S$.
We first consider the exceptional covering cases; see
\cite[Table~6.1.3]{GLS98} for the list of $S$.
Note that the (iBAW) condition has been verified
for the alternalting groups in \cite{Ma14},
for simple groups of Lie type in defining characteristic in \cite{Sp13},
and for cyclic blocks in \cite{KS16a,KS16b}.
Then by a similar argument in \cite[\S 8]{FLZ20},
the only prime we need to consider for the  simple group $\PSL_3(4)$, $\PSU_4(3)$, $\PSU_6(2)$ is just $3$, $2$, $3$, respectively.
These cases are settled in
\cite{Du19, Du20}. We mention that
the paper \cite{Du20} dealt with the the blocks of the universal covering groups of $\PSU_4(3)$ and $\PSU_6(2)$ which dominate no block of $\SU_4(3)$ and $\SU_6(2)$, while the blocks of special unitary groups are considered in this paper.

\begin{proof}[Proof of Theorem \ref{ibaw-nq-1-prime}]
By the above arguments, we may assume that $G=\SL_n(\eps q)$ is the universal covering group of $S$.
According to Proposition~\ref{prop:(2.3)} (iii), the block correspondence property (\ref{equ:block-corr}) holds.
For an  $\ell$-block $B$ of $G$, we have $\ell\nmid |\tG:\tG_B|$, where $\tG=\GL_n(\eps q)$.
On the other hand $|\tG:\tG_\psi|_{\ell'}$ is a product of pairwise distinct primes.
Thus $\gcd(|\tG:\tG_B|,|\tG_B:\tG_\psi|)=1$ and the assertion follows by  Theorem~\ref{iBAW-SLSU-some-special-case}~(v).
\end{proof}

\begin{cor}\label{ibaw-n-prime}
Assume that $n$ is square-free.
Then the  (iBAW)  condition holds for the simple group $S=\PSL_n(\eps q)$. 	
\end{cor}

\begin{proof}
By  \cite[Thm.~C]{Sp13}, we only need to the consider the non-defining characteristic case, which  follows from Theorem \ref{ibaw-nq-1-prime} immediately.
\end{proof}

Now we consider the simple groups of type $\mathsf A$ of small rank.
We first have the following.

\begin{lem}\label{ibaw-n-4}
The  (iBAW)  condition holds for $S=\PSL_4(\eps q)$ and any prime.
\end{lem}

\begin{proof}
As above, we assume that $G=\SL_4(\eps q)$ is the universal covering group of $S$.
Let $\tG=\GL_4(\eps q)$.
Also by \cite[Thm.~ C]{Sp13}
 we assume that $\ell\nmid q$.
If $\ell=2$, then $\tG_B=\tG$ for every $2$-block $B$ of $G$ and then the (iBAW) condition holds by  Theorem \ref{ibaw-nq-1-prime}.

Assume that $\ell$ is odd.
By Theorem \ref{ibaw-nq-1-prime},
we only need to consider the case $4\mid q-\eps$.
If $\ell\nmid q-\eps$,  then
we can check directly that the Sylow $\ell$-subgroups of $G$ are cyclic if $e>2$.
In addition, if $e=2$, then $G$ has an abelian defect group,
and  any $\ell$-block of $G$ is either a cyclic block or of maximal defect.
Recall that $e$ denotes the multiplicative order of $\eps q$ modulo $\ell$.
Then the lemma follows from \cite{KS16a} and Theorem \ref{uni-maxdef}.
Now we assume that $\ell\mid q-\eps$.
Let $B$ be an $\ell$-block of $G$ and $\tilde B$ an $\ell$-block of $\tG$ covering $B$.
By \cite{FS82}, we may assume that $\tilde B=\mathcal E_\ell(\tG,s)$ for some semisimple $\ell'$-element $s$ of $\tG$.

If $s$ has an elementary divisor of degree $4$, then it can be checked that $\tB$ is a cyclic block, and so is $B$.
If $s$ has an elementary divisor of degree $3$, then there is an irreducible character $\chi\in \mathcal E(\tG,s)$ such that $\Res^{\tG}_G\chi$ is irreducible, and thus the (iBAW) condition holds for $B$ by Proposition~\ref{prop:(2.3)} (ii) and
Theorem~\ref{iBAW-SLSU-some-special-case}~(i).
If $s$ has two (possibly the same) elementary divisors of degree $2$, then using the structure of defect groups  of $\tilde B$ giving in \cite{FS82} and the determinant of radical subgroups in \cite{FLZ20}, we know $B$ is a cyclic block.
Now we assume that $s$ has no elementary divisors of degree $\ge 3$ and has an elementary divisor of degree $1$.
By~\cite[Remark~4.13]{Feng19}, $\tilde B$ covers only one block of $G$.
Blocks of maximal defect are dealt with in Theorem~\ref{uni-maxdef}.
Let $D$ be a defect group of $B$ and $\tilde D$ be a defect group of $\tB$ with $\tilde D\cap G=D$.
If $B$ is not of maximal defect, then direct calculation shows that $\tilde D$ is abelian.
By \cite[Prop.~4.24]{FLZ20}, $C_{\tG}(D)=C_{\tG}(\tilde D)$.
If $G\le J\le\tG$ and $B_1$ is a block of $J$ covered by $\tB$, then $\tilde D\le C_{\tG}(C_{\tG}(\tilde D))= C_{\tG}(C_{\tG}(D))\le C_{\tG}(C_J(\tilde D\cap J)_{\ell'})$.
So $C_{\tG}(C_J(\tilde D\cap J)_{\ell'})G\ge \tilde DG=\tG$.
By \cite[Prop.~5.2]{CS15}, $\tG_{B_1}=\tG$.
Thus the (iBAW) condition holds for $B$ by Proposition~\ref{prop:(2.3)} (i) and
Theorem~\ref{iBAW-SLSU-some-special-case}~(i) and this completes the proof.
\end{proof}

By Corollary \ref{ibaw-n-prime}  and Lemma \ref{ibaw-n-4} we have a consequence for simple groups of type $\mathsf A$ with small rank immediately.

\begin{prop}\label{iBAW-small-rank}
Let $S\in\{\PSL_n(q),\PSU_n(q)\mid  n\le 7 \}$ be a simple group.
Then the (iBAW) condition	holds for $S$.
\end{prop}

\section{The blockwise Alperin weight conjecture for finite groups with abelian Sylow $3$-subgroups}\label{AWC-abel-sylow3}

We consider finite groups with abelian Sylow $3$-subgroups and
prove Theorem~\ref{abel-syl-3}.

\begin{proof}[Proof of Theorem  \ref{abel-syl-3}]
According to the reduction theorem \cite[Thm.~A]{Sp13}, it suffices to prove that any non-abelian simple group $S$ with order dividing by $\ell$ involved in $G$ satisfies the  (iBAW)  condition.
If $S$ is an alternalting group, then this follows by \cite[Thm.~1.1]{Ma14}.
If $S$ is one of the sporadic simple groups, then the  (iBAW)  condition has been checked in \cite{Br} except when $S$ is one of
$\J_4$, $\mathrm{Fi}_{24}'$, $\mathbb B$ and $\mathbb M$.
These  four sporadic simple groups are not involved in $G$ since their Sylow $3$-subgroups are non-abelian~(cf. \cite[\S 5.3]{GLS98}).

Now we assume that $S$ is of Lie type.
Then by \cite[Lemma~2.2]{FLL17b}, $S$ is $\PSL_n(\eps q)$ with  $n\le 5$, a Suzuki group, or $\PSp_4(q)$ ($q>2$, $3\nmid q$).
The case of defining characteristic has been verified in \cite[Thm.~C]{Sp13}. 
If $S$ is a Suzuki group, then $S$ satisfies the  (iBAW)  condition by \cite[Thm.~1.1]{Ma14}. The simple group $\PSp_4(q)$ is verified in~\cite{SF14} for even $q$
 and in \cite{BSF19} or \cite{LL19} for odd $q$, while the simple groups $\PSL_n(\eps q)$ with $n\le 5$ satisfy the  (iBAW)  condition by Proposition \ref{iBAW-small-rank}.
This completes the proof.
\end{proof}

 \section{$2$-blocks of special linear and unitary groups}\label{2-blocks-SLSU}

This section is a continuation of \cite[\S 4]{Feng19}, focusing on the blocks of special linear and unitary groups.
In \cite[\S 4]{Feng19}, the author gives a classification of the blocks of $\SL_n(\eta q)$ for odd prime $\ell \nmid q$, using the labelling set of $d$-Jordan-cuspidal pairs given in \cite{CE99, KM15}.
More precisely, for a given $\ell$-block $\tB$ of $\tG=\GL_n(\eta q)$,  \cite[Remark~4.13]{Feng19} gives the number  of $\ell$-blocks of $G=\SL_n(\eta q)$ covered by $\tB$ when $\ell$ is odd.
If $\ell=2$, \cite[Remark~4.13]{Feng19} also gives an upper bound for this number.
In this section we compute this number.

\subsection{Basic results}

For arbitrary finite groups $K\unlhd H$ and $\chi\in\Irr(H)$, we denote by $\kappa^H_K(\chi)$ the number of irreducible constituents of $\Res^H_K(\chi)$ forgetting multiplicities.
If $\tB$ is a block of $H$, then we denote the number of  blocks of $K$ covered by $\tB$ by $\kappa^{H}_{K}(\tB)$.

Let $G\leqslant\hG\leqslant\tG$.
Since $\tG/G$ is cyclic, by Clifford theory (see for example  \cite[Chap.~6]{Is76} or \cite[Lemma~2.1]{Feng19}) we have  
\begin{equation}\label{eq:branch-irr}
\kappa_{\hG}^{\tG}(\tchi_{s,\lambda}) = |\{ z\in\mathfrak Z\mid  z.(s,\lambda)^{\tG}=(s,\lambda)^{\tG}, o(z)\mid  |\tG/\hG|\}|.
\addtocounter{thm}{1}\tag{\thethm}
\end{equation}
Note that the group $\mathfrak Z$ is defined as on page \pageref{def-frakZ}.

Let $G,\hG,\tG$ be as above.
By \cite[Remark~4.7]{CS17}, for any $\tchi\in\Irr(\tG)$, there is a $\hchi_0\in\Irr(\hG\mid\tchi)$ such that $(\tG\rtimes{D})_{\hchi_0}=\tG_{\hchi_0}\rtimes D_{\hchi_0}$.
Inspired by this, we introduce the following.
\begin{defn}\label{defn:V-split}
	Let $H$, $V$ be finite groups such that $V$ normalises $H$.
	Assume that finite groups $K$ is a normal subgroup of $HV$ satisfy that  $H\cap V$ acts on $K$ via inner automorphisms.
	For any $\tchi\in\Irr(H)$, a character $\chi\in\Irr(K\mid\tchi)$ is called \emph{$V$-split} if $(HV)_\chi=H_\chi V_\chi$.
\end{defn}

An easy and immediate property is as follows.
\begin{lem}\label{lem:V-split}
	Keep the assumptions in Definition~\ref{defn:V-split}.
		\begin{enumerate}[\rm(i)]
		\item If $\chi\in\Irr(K\mid\tchi)$ is $V$-split, then $V_{\tchi}\leqslant V_\chi$.
	\item Assume furthermore that $V$ acts trivially on $H/K$.
	If one character in $\Irr(K\mid\tchi)$ is $V$-split, then so is any one in $\Irr(K\mid\tchi)$.
	\end{enumerate}
\end{lem}

\begin{proof}
	(i) is obvious. For (ii), we assume $\chi\in\Irr(K\mid\tchi)$ is $V$-split.
	Any character in $\Irr(K\mid\tchi)$ is of the form $\chi^h$ for some $h\in H$.
	Then $(HV)_{\chi^h}=(HV)_\chi^h$ and $H_{\chi^h}=H_\chi^h$.
	Assume $v\in V_\chi$, then $(\chi^h)^v=\chi^{v^{-1}hv}$.
	Since $V$ acts trivially on $H/K$, there is $k\in K$ such that $v^{-1}hv=kh$.
	Thus $(\chi^h)^v=\chi^h$.
	So $V_\chi\leqslant V_{\chi^h}$ and the equality holds by interchanging the role of $\chi$ and $\chi^h$.
	Consequently, $H_\chi^h V_\chi \leqslant (HV)_{\chi^h}$.
	Comparing the order, we have $(HV)_{\chi^h} = H_\chi^h V_\chi = H_{\chi^h} V_{\chi^h}$.
\end{proof}

The following lemma is a generalisation of \cite[Lemma~4.7]{Feng19}.

\begin{lem}\label{lem-split-centra}
	Let 
	$H$, $\tK$, $\tH$ be  finite groups such that $\tH=H\tK$ and $\tK\unlhd \tH$.
Assume that $\tK=\prod_{i=1}^{u} \tK_i^{t_i}$	with $\tK_i=\GL_{m_i}((\eta q)^{d_i})$.
Assume that $\iota_i:\tK_i\to \fZ$ is a group homomorphsm for every $i$
and set $\iota=\prod_{i=1}^u \iota_i^{t_i}:\tK\to \fZ$.
Let $K_i=\ker(\iota_i)$  and  $K=\ker(\iota)$.	
The group $H=\prod_{i=1}^u H_i\wr \fS_{t_i}$ satisfies that $H_i$ stabilizes $\tK_i$ and $H_i= K_i\langle \sigma_i \rangle$, where $\sigma_i$ acts on $\tK_i$ as $F_{\eta q}$.
	
Let $\tchi\in\Irr(\tK)$.
For any $\chi\in\Irr(K\mid \tchi)$, $\tilde l\in\tK$, $h\in H$,  
	one can have $\chi^{\tilde l}=\chi^h$ only if $\chi^{\tilde l}=\chi^h=\chi$.	
\end{lem}

\begin{proof}
It is equivalent to proving that every character	$\chi\in\Irr(K\mid \tchi)$ is $H$-split.
First we consider the case that $u=1$ and $t_1=1$, namely, $\tK=\GL_{m_1}((\eta q)^{d_1})$, $[\tK,\tK]\le K_1= K\le \tK$, $|\tK/K|\mid (q-\eta)$ and $H=K\langle\sigma_1 \rangle$.
Since $\sigma_i$ acts on $\tK_i$ as $F_{\eta q}$ a field automorphisms of order $d_i$, by \cite[Remark 4.7]{CS17}, there is a $\chi_0\in\Irr(K\mid \tchi)$ which is $H$-split.
Also it can be checked directly that $\sigma_1$ acts trivially on $\tK/K$.  
Thus by Lemma \ref{lem:V-split}, every $\chi\in\Irr(K\mid \tchi)$ is $H$-split.

Now let $u=1$ and $t_i>1$, namely, $\tK=\tK_1^{t_1}$ with $\tK_1=\GL_{m_1}((\eta q)^{d_1})$.
Let $\tchi=\tchi_1\times \cdots\times\tchi_{t_1}$, where $\tchi_i\in\Irr(\tK_1)$ for $1\le i\le t_1$.
For $\chi\in\Irr(K\mid \tK)$, we let
$\chi_0\in\Irr(K_1^{t_1}\mid \chi)$ with $\chi_0=\chi_{0,1}\times \cdots\times\chi_{0,t_1}$, where $\chi_{0,i}\in\Irr(K_1)$ for $1\le i\le t_1$.
Let $h\in H$ with $\chi^h\in\Irr(K\mid \tchi).$
Then
$\chi_0^h\in\Irr(K_1^{t_1}\mid\tchi)$,
and without loss of generality, we may assume that $h=(h_1,\ldots,h_{t_1};\tau)$ with $h_i\in H_1$ and $\tau=(1,\ldots,t_1)$ is a $t_1$-cycle.
So $\chi_0^h=\chi_{0,t_1}^{h_{t_1}}\times \chi_{0,1}^{h_1}\times\cdots \chi_{0,t_1-1}^{h_{t_1-1}}$.
Hence
there exist $\tilde l_1,\ldots, \tilde l_{t_1}\in \tK_1$ such that
$\chi_{0,1}^{\tilde l_1}=\chi_{0,t_1}^{h_{t_1}}$ and $\chi_{0,i}^{\tilde l_i}=\zeta_{0,i-1}^{h_{i-1}}$ for $2\le i\le t_1$.
Since the outer automorphisms induced by $H_1\tK_1$ on $K_1$ form an abelian group, one has that $\chi_{0,1}^{\tilde l_1\cdots \tilde l_{t_1}}=\chi_{0,1}^{h_1\cdots h_{t_1}}$, and then by the above  paragraph
$\tilde l_1\cdots \tilde l_{t_1}\in (\tK_1)_{\chi_{0,1}}$.
By replacing $\tilde l_1$ with $(\tilde l_1\cdots \tilde l_{t_1})^{-1}\tilde l_1$, we can assume that $\tilde l_1\cdots \tilde l_{t_1}=1$.
Let $\tilde l=(\tilde l_1,\ldots, \tilde l_{t_1})$.
Then $\tilde l\in K$ and $\chi_0^h=\chi_0^{\tilde l}$.
So 
$\chi_0^h\in\Irr(K_1^{t_1}\mid \chi)$.
Hence $\chi,{^h\chi}\in\Irr(K\mid \tchi)\cap\Irr(K\mid\chi_0)$. By \cite[Lemma 7.2]{FLZ20},
$\chi^h=\chi$.
This proves that $\chi$ is $H$-split.

The assertion in the general case now follows by reduction to the preceding cases.
\end{proof}

Assume that $p$ is odd and $\ell=2$ and let $\alpha$ be a non-negative integer, then the sets $\cF_{\alpha,0},\cF_{\alpha,1},\cF_{\alpha,2},\cF_\alpha,\cF'_\alpha$ are defined similarly as $\cF_{0},\cF_{1},\cF_{2},\cF,\cF'$
with $q$ replaced by $(\eta q)^{2^\alpha}$ respectively.
We have a surjective map
\begin{equation}\label{eq:Phi_alpha}
\Phi_\alpha:\ \cF_\alpha \to \cF,\ \group{F_{\eta{q}}^{2^\alpha}}\cdot\xi \mapsto \group{F_{\eta{q}}}\cdot\xi.
\addtocounter{thm}{1}\tag{\thethm}
\end{equation}
The inverse images of $\group{F_{\eta{q}}}\cdot\xi$ under $\Phi_\alpha$ are the $\group{F_{\eta{q}}}$-orbits of the set $\group{F_{\eta{q}}^{2^\alpha}}\cdot\xi$.

Let $\Ga\in\cF$ and $\xi\in\barF_{q}^\times$ be a root of $\Ga$.
We define $\Ga_{(\al)}$ to be a polynomial in $\cF_\al$ which has a root $\xi$.
Then  the set of roots of $\Ga_{(\al)}$ is an element in the  inverse images of $\group{F_{\eta{q}}}\cdot\xi$ under $\Phi_\alpha$. Conversely, every element in the  inverse images of $\group{F_{\eta{q}}}\cdot\xi$ under $\Phi_\alpha$ provides a choice of $\Ga_{(\al)}$.

\begin{lem}\label{number-2-pow}
Let $z\in\mathfrak Z$ and $\Ga\in\cF'$ (i.e., the roots of $\Ga$ have $2'$-order).
Then $z.\Ga=\Ga$ if and only if $z.\Ga_{(\al)}=\Ga_{(\al)}$ for any non-negative integer $\alpha$.
\end{lem}

\begin{proof}
Let $\xi$ be a root of both $\Ga$ and $\Ga_{(\al)}$, and
 $o(\xi)=p_1^{t_1}\cdots p_u^{t_u}$, where $p_1,\ldots,p_u$ are distinct odd primes. Also we write $d=d_\Ga=m_\Ga 2^{\al_\Ga}$ with odd $m_\Ga$.
We denote by $\nu_{i}$ the discrete valuation such that $\nu_{i}(p_i)=1$,
by $e_i$ the multiplicative order of $\eps q$ modulo $p_i$,
 and by $a_{i}=\nu_{i}((\eta q)^{e_i}-1)$ for $1\le i\le u$.
 Let $\alpha_i=\max \{ t_{i}-a_{i}, 0 \}$.
Then by \cite[(3A)]{FS82} or \cite[Lemma 5.3]{FLZ20}, $d$ is the least common multiple of the
$e_i p_i^{\al_i}$, where $i$ runs through the integers between $1$ and $u$.
So there exists $1\le i_0\le u$ such that $2^{\alpha_\Ga}\mid e_{i_0}$.
Also, the polynomial $\Ga_{(\al)}$ has degree $d/\gcd(d,2^\alpha)$.
Thus $\Ga_{(\al)}=\Ga$ if $\al_\Ga=0$ and $\Ga_{(\al)}=\Ga_{(\al_\Ga)}$ if $\al\ge \al_\Ga$.
So it suffices to let $0<\al\le\al_\Ga$.
In particular $e_{i_0}>1$.

We may assume that $z$ has $2'$-order. 
It is obvious that $z.\Ga_{(\al)}=\Ga_{(\al)}$ implies $z.\Ga=\Ga$.
Now we assume that $z.\Ga=\Ga$,
then  $z\xi=\xi^{(\eta q)^k}$ and then
$z=\xi^{(\eta q)^k-1}$ for some $k\ge 1$.
Thus $o(\xi)\mid ((\eta q)^k-1)(\eta q-1)$ and in particular $p_{i_0}^{t_{i_0}}\mid ((\eta q)^k-1)(\eta q-1)$.
Note that $p_{i_0}\nmid (\eta q-1)$ since
$e_{i_0}>1$.
So $p_{i_0}^{t_{i_0}}\mid (\eta q)^k-1$ and thus $e_{i_0}\mid k$.
In particular, $2^{\al_\Ga}\mid k$.
By the definition of $\Ga_{(\al)}$, $\xi^{(\eta q)^k}$ is a root of $\Ga_{(\al)}$, and so $z.\Ga_{(\al)}=\Ga_{(\al)}$. This completes the proof.
\end{proof}

\subsection{$2$-blocks of special linear and unitary groups}

We will follow the notation for basic 2-subgroups and radical $2$-subgroups of linear and unitary groups in \cite{FLZ20} and always consider the $2$-modular representations from now on.

Let $\tB$ be a $2$-block of $\tG$ with a maximal Brauer pair
$(\tR,\tilde b)$. Let $\ttheta$ be the canonical character of $\tilde b$
and $R=\tR\cap G$.
Note that the defect groups are  radical subgroups.

First we recall the structure of $\tR$.
Set $2^{a+1}=(q^2-1)_2$, thus $a\geqslant2$.
Recall that for  $\tG=\GL_n(\eta q)$ with odd $q$, by \cite{Brou86}, there is a semisimple $2'$-element $s$ of $\tG$ associated with $\tB$, namely, $\tB=\mathcal{E}_2(\tG,s)=\bigcup\limits_{t}\mathcal{E}(\tG,st)$,
where $t$ runs through the semisimple $2$-elements of $\tG$ commuting with $s$. 
Let $s=\prod\limits_{\Gamma\in\cF} m_\Gamma(s)(\Gamma)$ be the primary decomposition.
Here $m_\Gamma(s)$ is the mutiplicity of $\Gamma$  in $s$ as an elementary divisor and $(\Gamma)$ is the companion matrix of $\Gamma$.
By \cite{Brou86} again, a Sylow $2$-subgroup $\tR$ of $C_{\tG}(s)$ is a defect group of $\tilde B_s$.
Now $C_{\tG}(s)=\prod\limits_\Gamma C_\Gamma(s)$ with $C_\Gamma(s)\cong\GL_{m_\Gamma(s)}((\eta q)^{d_\Gamma})$.
Then $\tR=\prod\limits_\Gamma \tR_\Gamma$, where $\tR_\Gamma$ is a Sylow $2$-subgroup of $C_\Gamma(s)$.
For each $\Gamma$, let $m_\Gamma$ and $\alpha_\Gamma$ be such that $m_\Gamma2^{\alpha_\Gamma}=d_\Gamma$ with odd $m_\Gamma$.
Then by \cite{CF64}, we have the following.

\begin{enumerate}[(1)]
	\item If $4\mid q-\eta$ or $\alpha_\Gamma\geqslant1$, set $m_\Gamma(s)=2^{\beta_{\Gamma,1}}+\cdots+2^{\beta_{\Gamma,t_\Gamma}}$ with $0\leqslant\beta_{\Gamma,1}<\cdots<\beta_{\Gamma,t_\Gamma}$.
	Then $\tR_\Gamma=\prod\limits_{i=1}^{t_\Gamma} \tR_{m_\Gamma,\alpha_\Gamma,0,\bar\beta_{\Gamma,i}}$, where $\tR_{m_\Gamma,\alpha_\Gamma,0,\bar\beta_{\Gamma,i}}$ is defined as in \cite[\S 4]{FLZ20} and $\bar\beta_{\Gamma,i}=(1,1,\ldots,1)$ with $\beta_{\Gamma,i}$ one's.
	
	\item If $4\mid q+\eta$ and $\alpha_\Gamma=0$, then $d_\Gamma=m_\Gamma$ is odd.
	\begin{enumerate}
		\item[(2i)] If furthermore $m_\Gamma(s)$ is even, set $m_\Gamma(s)=2^{\beta_{\Gamma,1}}+\cdots+2^{\beta_{\Gamma,t_\Gamma}}$ with $0<\beta_{\Gamma,1}<\cdots<\beta_{\Gamma,t_\Gamma}$.
		Then $\tR_\Gamma=\prod\limits_{i=1}^{t_\Gamma}\tilde S_{m_\Gamma,1,0,\overline{\beta_{\Gamma,i}-1}}$, where $\tilde S_{m_\Gamma,1,0,\overline{\beta_{\Gamma,i}-1}}$ is defined as in~\cite[\S 4]{FLZ20}.
		
		\item[(2ii)] If furthermore $m_\Gamma(s)$ is odd, set $m_\Gamma(s)=1+2^{\beta_{\Gamma,1}}+\cdots+2^{\beta_{\Gamma,t_\Gamma}}$ with $0<\beta_{\Gamma,1}<\cdots<\beta_{\Gamma,t_\Gamma}$.
		Then $\tilde R_\Gamma=\tR_{m_\Gamma,0}\times\prod\limits_{i=1}^{t_\Gamma} \tilde S_{m_\Gamma,1,0,\overline{\beta_{\Gamma,i}-1}}$. Note that $\tR_{m_\Gamma,0}=\{\pm I_{m_\Gamma}\}$. 
	\end{enumerate}
\end{enumerate}

Let $\tC=C_{\tG}(\tR)$. Then $\tC\cong\prod_\Gamma \prod_{i=1}^{t_\Ga} \GL_{m_\Ga}((\eta q)^{2^{\al_\Ga}})\otimes I_{2^{\beta_{\Ga,i}}}$ and $\ttheta=\prod_\Gamma \prod_{i=1}^{t_\Ga} \ttheta_\Ga\otimes I_{2^{\beta_{\Ga,i}}}$, where $\ttheta_\Ga\otimes I_{2^{\beta_{\Ga,i}}}$ is defined as in \cite[\S 3E]{FLZ20}.
Note that $\ttheta_\Ga$ is the character of $\GL_{m_\Ga}((\eta q)^{2^{\al_\Ga}})$ corresponding to the semisimple $2'$-element with the only elementary divisor $\Ga_{(\alpha_\Ga)}$ of multiplicity 1.

Now we recall the normaliser $\tN=N_{\tG}(\tR)$.
First note that it is possible that some component of $\tR_\Ga$ is isomorphic to a component of $\tR_{\Ga'}$ for different $\Ga$, $\Ga'$ with $d_\Ga=d_{\Ga'}$.
From this we rewrite $\tR=\prod_{i=1}^{t}\tR_i^{u_i}$, where $\tR_i$'s are the components of $\tR$ such that $\tR_i\ne \tR_j$ if $i\ne j$.
Then $\tC=\prod_{i=1}^{t}\tC_i^{u_i}$  and
$\tN=\prod_{i=1}^{t} N_i\wr \fS_{u_i}$, where $\fS_{u_i}$ is the symmetric group on $u_i$ symbols.
Now we assume that $\tR_i$ is a component of $\tR_\Ga$ for some $\Ga$. 
By \cite{An92,An93}, then $\tN_i/\tR_i$ is a direct product of $\tC_i\tR_i/\tR_i$ and a subgroup of $\tN_i/\tR_i$ if $\alpha_\Ga=0$.
If $\alpha_\Ga>0$, then $\tC_i\cong\GL_{m_\Ga}((\eta q)^{2^{\al_\Ga}})$.
Then $\tN_i/\tR_i$ is a direct product of $\tN_{i,0}/\tR_i$ and a subgroup of $\tN_i/\tR_i$, where $\tN_{i,0}\unlhd \tN_i$ such that $\tN_i=\tN_{i,0}\langle \tau \rangle$.
Here $\tau$ acts on $\tC_i\cong\GL_{m_\Ga}((\eta q)^{2^{\al_\Ga}})$ as $F_{\eta q}$.

Let  $R=\tR\cap G$ and
$\tR'=\mathrm{O}_2(N_{\tG}(R))$. Then $R\le\tR'\le \tR$ and $N_{\tG}(\tR')=N_{\tG}(R)$ by~\cite[Lemma~2.2]{Feng19}.
By \cite[Prop. 4.37 and 4.47]{FLZ20},
$C_{\tG}(\tR')=C_{\tG}(R)$.

First we consider the case that $\tR'=\tR$.
We view $\ttheta$ as a character of $C_{\tG}(\tR)/Z(\tR)$
and let $\theta\in\Irr( C_{G}(R)/Z(R) \mid \ttheta)$.
Then $\theta$ is also of $2$-defect zero.
We also view $\theta$ as a character of $RC_G(R)$ whose kernel contains $R$.
Let $b$ be the $2$-block of $RC_{G}(R)$ containing $\theta$.
Then
$\theta$ is the canonical character of $b$.
Let $B=b^G$.
By~\cite[Lemma~2.3]{KS15}, $\tB$ covers $B$.
Conversely, every $2$-block of $G$ covered by $\tB$ can be obtained by the above process from a suitable $\theta\in \Irr( C_{G}(R)/Z(R) \mid \ttheta^{\tilde n})$ with $\tilde n\in N_{\tG}(\tR)$.

If $\tR'\ne\tR$, then according to \cite[\S 5D]{FLZ20}, exactly one of the following cases occurs.
We set $\tC'=C_{\tG}(\tR')$ and $\tN'=N_{\tG}(\tR')$.

\emph{Case (i)}.  $4\mid q-\eta$, $\tR=\tR_\Gamma=\tR_{m_\Gamma,\alpha_\Gamma}$ and $\ttheta=\ttheta_\Ga$
for some $\Gamma\in\cF'$ with $\alpha_\Gamma>0$.
This means that $m_\Ga(s)=1$ and $n=d_\Ga$. 
Up to conjugacy,  we may take $\tR'=\tR_{m_\Gamma\ell^{\alpha'},\alpha_\Gamma-\alpha'}$, where $\al'=\min\{a,\al_\Ga\}$.
Let $\ttheta'=\ttheta'_\Ga=\ttheta_\Ga^{(\alpha')}$, where $\ttheta_\Ga^{(\alpha')}$ is defined as in \cite[\S 5C]{FLZ20}.
Also, $\tC'\cong\GL_{m_\Gamma\ell^{\alpha'}}((\eta q)^{2^{\alpha_\Gamma-\alpha'}}) $ and $\tN'=\tC'\langle \tau \rangle$, where $\tau$ acts on $\tC'$ as $F_{\eta q}$.

\emph{Case (ii)}. $4\mid q-\eta$, $a=2$, $\tR=\tR_\Gamma=\tR_{m_\Gamma,0,0,\bar{1}}$ and $\ttheta=\ttheta_{\Ga,1}$ 
for some $\Gamma\in\cF'$ with $\alpha_\Gamma=0$, where $\ttheta_{\Ga,1}$ is defined as in \cite[\S 3E]{FLZ20}.
This means that $m_\Ga(s)=2$ and $n=2d_\Ga$. 
Up to conjugacy,  we may take $\tR'=\tR_{m_\Gamma,0,1}$. By \cite[\S 3B]{FLZ20}, $\tN'/\tR'=\tC'\tR'/\tR'\times \Sp_2(2)$ and $\tC'\cong \GL_{m_\Ga}(\eta q)\otimes I_2$.
Let $\ttheta'=\ttheta_\Ga\otimes I_2$.

\emph{Case (iii)}.  $4\mid q+\eta$, $\tR=\tR_\Gamma=\tR_{m_\Gamma,\alpha_\Gamma}$ and $\ttheta=\ttheta_\Ga$
 for some $\Gamma\in\cF'$ with $\alpha_\Gamma>1$.
This means that $m_\Ga(s)=1$ and $n=d_\Ga$.
Up to conjugacy,  we take $\tR'=\tR_{2m_\Gamma,\alpha_\Gamma-1}$ and $\ttheta'=\ttheta'_\Ga=\ttheta_\Ga^{(1)}$.
Also, $\tC'\cong\GL_{2m_\Gamma}((\eta q)^{2^{\alpha_\Gamma-1}}) $ and $\tN'=\tC\langle \tau \rangle$, where $\tau$ acts on $\tC'$ as $F_{\eta q}$.

\emph{Case (iv)}. $4\mid q+\eta$, $a=2$ and $\tR=\tR_\Gamma=\tilde S_{m_\Gamma,1,0}$ for some $\Gamma\in\cF'$ with $\alpha_\Gamma=0$.
This means that $m_\Ga(s)=2$ and $n=2d_\Ga$.
Up to conjugacy,  we may take $\tR'=\tR^-_{m_\Gamma,0,1}$, where $\tR^-_{m_\Gamma,0,1}$ is defined as in \cite[\S 3C]{FLZ20}.
Also $\tN'/\tR'=\tC'\tR'/\tR'\times \mathrm{GO}^-_2(2)$  and $\tC'\cong \GL_{m_\Ga}(\eta q)\otimes I_2$.
Let $\ttheta'=\ttheta_\Ga\otimes I_2$.

For the case (i)--(iv) above, the Brauer pair $(\tR',\ttheta')$ is a $\tB$-pair of $\tG$ by \cite{Brou86}; see also \cite[\S 5D]{FLZ20}.
We view $\ttheta'$ as a character of $C_{\tG}(\tR')/Z(\tR')$
and let $\theta\in\Irr( C_{G}(R)/Z(R) \mid \ttheta')$.
Then $\theta$ is also of $2$-defect zero by \cite[\S 5.C]{FLZ20}.
We also view $\theta$ as a character of $RC_G(R)$ whose kernel contains $R$.
Let $b$ be the $2$-block of $RC_{G}(R)$ containing $\theta$.
Then
$\theta$ is the canonical character of $b$.
Let $B=b^G$.
By \cite[Lemma~2.3]{KS15}, $\tB$ covers $B$.
Conversely, every $2$-blocks of $G$ covered by $\tB$ can be obtained through the above 	process from a suitable $\theta\in \Irr( C_{G}(R)/Z(R) \mid \ttheta'^{\tilde n})$ with $\tilde n\in N_{\tG}(\tR')$.

In order to deal with the two cases that $\tR'=\tR$ and $\tR'\ne \tR$ simultaneously,
we also use the notation $\tC'$, $\tN'$, $\ttheta'$ for $\tC$, $\tN$, $\ttheta$ if $\tR'=\tR$.
Also, we let $C=C_G(R)$, $N=N_G(R)$.

\begin{lem}\label{lem-block}
	Keep the hypotheses and setup above.
Then $\kappa^{\tG}_{G}(\tB_s)=\kappa^{\tC'}_{C}(\ttheta')$.
\end{lem}

\begin{proof}
As above, every $2$-block of $G$ covered by $\tB$ can be obtained through the above 	process from a suitable $\theta\in \Irr( C_{G}(R)/Z(R) \mid \ttheta'^{\tilde n})$ with $\tilde n\in N_{\tG}(\tR')$.		
If we are neither in the case (ii) nor (iv), then using \cite[Prop.~4.34 and 4.43]{FLZ20},
direct calculation case-by-case shows that $\tN'=\tR'\tC' N$, and then we may assume that $\tilde n\in N$.
If we are in the case (ii) or (iv), then $\tN'$ acts on $\tC'$ by inner automorphisms.
Hence if we fix a $\ttheta'$, then $\theta$ gives all 2-blocks of $G$ when  running through  $\Irr(C_{G}(R)/Z(R) \mid \ttheta')$.
From this it remains to prove that the characters in $\Irr(C_{G}(R)/Z(R) \mid \ttheta')$ are not $N$-conjugate.

By the above arguments, the group $\tC'=\prod_i (\tC_i')^{u_i}$ such that each component $\tC_i'$ is a general linear or unitary group and $\tN'=\prod_i \tN_i'\wr \fS_{u_i}$.
Furthermore, the action of any element of $\tN'$ on $\tC'$ (via conjugation) is a product of an inner automorphism and a field automorphism which commute.
Thus this assertion follows by Lemma~\ref{lem-split-centra}.
\end{proof}

\begin{thm}\label{num-block}
$\kappa^{\tG}_{G}(\tB_s)$ is the number of $z\in\rO_{2'}(\mathfrak Z)$ satisfying $z.\Gamma=\Gamma$ for every elementary divisor $\Gamma$ of $s$.
\end{thm}

\begin{proof}
According to \cite[Lemma 4.35 and 4.45]{FLZ20}, $\rO_{2'}(\tC'/C)\cong \rO_{2'}(\tG/G)$.
By \cite[Lemma~2.1]{Feng19} and Lemma \ref{lem-block}, $\kappa^{\tG}_{G}(\tB_s)$ is the number of
$z\in\mathfrak Z$ satisfying $\hat z \ttheta'=\ttheta'$.
Then the proof is similar with \cite[Cor.~4.12]{Feng19}, and then $\hat z \ttheta'=\ttheta'$ if and only if $z.\Ga_{(k_\Ga)}=\Ga_{(k_\Ga)}$ for every elementary divisor $\Gamma$ of $s$, where $k_\Ga$ is an integer such that $0\le k_\Ga\le\al_\Ga$.
Thus by Lemma \ref{number-2-pow}, 	$\hat z \ttheta'=\ttheta'$  if and only if $z\in\rO_{2'}(\mathfrak Z)$ and $z.\Gamma=\Gamma$ for every elementary divisor $\Gamma$ of $s$.	
\end{proof}

\begin{rem}\label{blocksofslsu}
	Analogously with \cite[Remark~4.13]{Feng19},
	we give a description for the $2$-blocks of $G=\SL_n(\eta q)$ by summarizing the argument above.

	For $\sigma\in\overline{\mathbb F}_q^\times$, we denote by $[\sigma]$ the set of all roots of the polynomial in $\cF$ which has $\sigma$ as a root.
	Thus $[\sigma]$ is a single $F_{\eta q}$-orbit.
	Denote by $\mathrm{deg}(\sigma)$ the cardinality of $[\sigma]$.
	Then $\mathrm{deg}(\sigma)$ is the minimal integer $d$ such that $\sigma^{(\eta q)^d-1}=1$ 
	and $$[\sigma]=\{\ \sigma, \sigma^{\eta q}, \sigma^{(\eta q)^2},\ldots,  \sigma^{(\eta q)^{\mathrm{deg}(\sigma)-1}}    \ \}.$$

	We call an $a$-tuple 
$$
	(([\sigma_1], m_1),\dots,([\sigma_a], m_a))
$$
	of pairs
	an \emph{$(n,2)$-admissible block tuple},
	if
	\begin{itemize}
		\item for every $1\le i\le a$, $\sigma_i\in \overline{\mathbb F}_q^\times$ is a $2'$-element, and
		$m_i$ is a positive integer, 
		\item $[\sigma_i]\ne[\sigma_j]$ if $i\ne j$, and
		\item $\sum\limits^{a}_{i=1}m_i\mathrm{deg}(\sigma_i)=n$.
	\end{itemize}
	
	An equivalence class of an $(n,2)$-admissible block tuple
	$$
	(([\sigma_1], m_1),\dots,([\sigma_a], m_a))
	$$
	 up to a permutation of pairs $([\sigma_1], m_1), \ldots, ([\sigma_a], m_a)$ is called an \emph{$(n,2)$-admissible block symbol} and is denoted as 
	\begin{equation}\label{admissiblesyms}
	\mathfrak b=[([\sigma_1],  m_1),\dots,([\sigma_a], m_a)].
	\addtocounter{thm}{1}\tag{\thethm}
	\end{equation}
	
	The set of $(n,2)$-admissible block symbols is in bijection with conjugacy classes of semisimple $2'$-elements of $\tG=\GL_n(\eta q)$,
	and then is a labeling set for 2-blocks of $\tG$ by~\cite{Brou86}.
	Denote by $\tB_{\mathfrak b}$ the 2-block of $\tG$ corresponding to the $(n,2)$-admissible block symbol $\mathfrak b$ as in (\ref{admissiblesyms}).
	Now we define $\kappa(\mathfrak b)$ to be the cardinality of the set
	$\{~z\in\rO_{2'}(\mathfrak Z)\mid [z\sigma_i]= [\sigma_i]\ \text{for all}\ 1\le i\le a~ \}.$
	
	By Theorem \ref{num-block}, $\kappa^{\tG}_G(\tB_{\mathfrak b})=\kappa(\mathfrak b)$
	~(\emph{i.e.} the number of 2-blocks of $G$ covered by $\tB_{\mathfrak b}$ is $\kappa(\mathfrak b)$).
	For two $(n,2)$-admissible block symbols $\mathfrak b$ and $\mathfrak b'$, if they are in the same $\rO_{2'}(\mathfrak Z)$-orbit, then the sets of 2-blocks of $G$ covered by $\tB_{\mathfrak b}$ and $\tB_{\mathfrak b'}$ are the same.
	
	If moreover, we let $(B_{\mathfrak b})_1$, $(B_{\mathfrak b})_2,\dots, (B_{\mathfrak b})_{\kappa(\mathfrak b)}$ be
	the 2-blocks of $G$ covered by $\tB_{\mathfrak b}$,
	then	the set $\{(B_{\mathfrak b})_j\}$, where $\mathfrak b$ runs through the $\rO_{2'}(\mathfrak Z)$-orbit representatives of $(n,2)$-admissible block symbols and $j$ runs through the integers between $1$ and  $\kappa(\mathfrak b)$,
	is a complete set of $2$-blocks of $G$.	
\end{rem}

\begin{rem}
We mention that the case that $\ell$ is odd can be also obtained by the same method as above, which has been dealt with in 
\cite[Remark~4.13]{Feng19} via the $d$-Jordan-cuspidal pairs.
\end{rem}

\begin{rem}
In \cite[Remark~4.13]{Feng19}, the author gives an upper bound for the number of 2-blocks of $\SL_n(\eta q)$ covered by a given 2-block of $\GL_n(\eta q)$.
In fact, if $4\mid q-\eta$, then that upper bound is just the $\kappa(\mathfrak b)$ defined in Remark \ref{blocksofslsu}, while if $4\mid q+\eta$,  that upper bound may be bigger than $\kappa(\mathfrak b)$ above.
This provides an example to show that the assumption of odd primes is necessary in \cite[Thm.~A~(e)]{KM15}, which states that the $d$-Jordan-cuspidal pairs form a labeling set for the blocks of a finite groups of Lie type.
\end{rem}

\subsection*{Acknowledgement}
We thank Gunter Malle for useful comments on an earlier version.


\end{document}